\documentclass[11pt]{article}
\usepackage{amsfonts}
\usepackage{amsmath}
\usepackage{color}
\usepackage{graphicx}

\newtheorem{theorem}{Theorem}

\newtheorem{definition}[theorem]{Definition}

\newtheorem{lemma}[theorem]{Lemma}

\newenvironment{proof}[1][Proof]{\noindent\textbf{#1.} }{\ \rule{0.5em}{0.5em}}
\begin{document}

\title{Existence result for impulsive coupled systems on the half-line\thanks{%
This work was supported by National Funds through Funda\c{c}%
\~{a}o para a Ci\^{e}ncia e Tecnologia (FCT), project
UID/MAT/04674/2013 (CIMA)}}
\author{Feliz Minh\'{o}s$^{(\dag )}$ and Robert de Sousa$^{(\ddag )}$ \\
$^{(\dag )}${\small Departamento de Matem\'{a}tica, Escola de Ci\^{e}ncias e
Tecnologia,}\\
{\small Centro de Investiga\c{c}\~{a}o em Matem\'{a}tica e Aplica\c{c}\~{o}%
es (CIMA),}\\
{\small \ Instituto de Investiga\c{c}\~{a}o e Forma\c{c}\~{a}o Avan\c{c}ada,
}\\
{\small \ Universidade de \'{E}vora. Rua Rom\~{a}o Ramalho, 59, }\\
{\small \ 7000-671 \'{E}vora, Portugal}\\ {fminhos@uevora.pt}\\
$^{(\ddag )}${\small Faculdade de Ci\^{e}ncias e Tecnologia,}\\
{\small N$\acute{u}$cleo de Matem\'{a}tica e Aplica\c{c}\~{o}es (NUMAT),}\\
{\small Universidade de Cabo Verde. Campus de Palmarejo,}\\
{\small 279 Praia, Cabo Verde}\\ {robert.sousa@docente.unicv.edu.cv}}
\date{}
\maketitle

\begin{abstract}
This work considers a second order impulsive coupled system of
differential equations with generalized jump conditions
in half-line, which can depend on the impulses of the unknown
functions and their first derivatives.

The arguments apply the fixed point theory, Green's functions \newline technique,
$L^{1}$-Carath\'{e}odory functions and sequences and
Schauder's fixed point theorem.

The method is based on Carath\'{e}odory concept of functions and sequences,
together with the equiconvergence on infinity and on each impulsive
moment, and it allows to consider coupled fully
nonlinearities and very general impulsive functions.

\bigskip

\end{abstract}

\section{Introduction}

Boundary-value problems in unbouded domains,
can be applied to a large variety of contexts,
(see for instance, \cite{BF, EKT, HLX, LGe1, Yuji, LHW, MR, PG, WLW}).

The theory of impulsive differential equations describes processes
in which a sudden change of state occurs at certain moments.
Some examples: in \cite {AKS}, Dishliev et al.,
make a very complete explanation of these equations,
where we can also find applications about
pharmacokinetic model,  logistic model,
Gompertz model (mathematical model for a time series),
Lotka-Volterra model and population dynamics. In \cite{G1}, Guo uses the fixed point
theory to investigate the existence and
uniqueness of solutions of two-point boundary value problems for second order
non-linear impulsive integro-differential equations on infinite intervals in a Banach
space. The same author, in \cite{G2}, by a comparison  result, obtains the existence of maximal and
minimal solutions of the initial value problems for a class of second-order
impulsive integro-differential equations in a Banach space. In \cite{LEE+LIU},
Lee and Liu study the existence
of extremal solutions for a class of singular boundary value problems of
second order impulsive differential equations. In \cite{FC1},  Minh\'{o}s,
and Carapinha study separated impulsive problems with a fully third
order differential equation, including an increasing homeomorphism, and
impulsive conditions given by generalized functions. In \cite{PLC},
Pang et al., consider a second-order impulsive differential equation with integral
boundary conditions, where they proposed some sufficient conditions for the existence of solutions,
by using the method of upper and lower solutions and Leray-Schauder degree theory.
In \cite{EYH}, Lee and Lee combine the method of upper and lower solutions with fixed point
index theorems on a cone to study the existence of positive solutions for a singular two point
boundary value problem of second order impulsive equation with fixed moments.

In \cite{WZL}, Wang, Zhang and Liang, consider the initial value problem for second order impulsive
integro-differential equations, which nonlinearity depend on the first derivative,
in a Banach space $E$:

\begin{equation*}
\left\{
\begin{array}{l}
x^{\prime \prime }(t)=f(t,x(t),x^{\prime }(t), Tx(t),Sx(t)),\;\;t\neq t_k,\;k=1,2,...m,\\
\Delta x(t_{k})=I_{k}(x(t_{k}),x^{\prime }(t_{k})),\;\;k=1,2,...m, \\
\Delta x^\prime(t_{k})=\overline{I_{k}}(x(t_{k}),x^{\prime }(t_{k})),\;\;k=1,2,...m,\\
x(0)=x_0,\;\;x^\prime(0)=x_0^*,
\end{array}%
\right.
\end{equation*}%
where $f \in C[J\times E^4,E]$, $J=[0,1]$, $0<t_0<t_1<...<t_k<...<t_m<1$.
$I_k,\,\overline{I_{k}} \in C[E^2, E]$, $k=1,2,...m$, $x_0, x_0^* \in E$,
$\theta$ denotes de zero element of $E$, $J^\prime =J\setminus\{t_1,t_2,...,t_m\}$ and $J_0=[0,t_1]$,
$J_k=(t_k,t_{k+1}]$, $k=1,2,...m$, $t_{m+1}=1$,
\begin{equation*}
Tx(t)=\int_0^t k(t,s)x(s)ds,\;\;\;Sx(t)=\int_0^1 h(t,s)x(s)ds,\;\;\forall t \in J,
\end{equation*}
where $k\in C[D, \mathbb{R}_+]$, $D=\{(t,s):J\times J|t\geq s\}$, $h \in C[J\times J,  \mathbb{R}_+]$, $\mathbb{R}_+=[0,+\infty)$.

In \cite{Kaufmann}, we found the study of second-order nonlinear differential equation

\begin{equation*}
\begin{array}{c}
{}(p(t)u^\prime(t))^\prime=f(t,u(t)),\;\;\;t \in (0,\;\infty)\backslash \{t_1,t_2,...,t_n\},\\
\end{array}%
\end{equation*}%
where $f:\left[ 0,+\infty\right) \times \mathbb{R}\rightarrow \mathbb{R}$ is continuous,
$p \in [ 0,+\infty)\cap C(0,+\infty)$ and $p(t)\geq0$ for all $t>0$, with the impulsive conditions
\begin{equation*}
\Delta u^\prime(t_k)=I_k(u(t_k)),\;\;k=1,...n,
\end{equation*}%
where $I_k:\mathbb{R}\rightarrow \mathbb{R}$, $k=1,...,n$, are Lipschitz continuous,
$n\geq1$, and the boundary conditions
\begin{equation*}
\begin{array}{c}
\alpha u(0)-\beta \lim_{t\rightarrow 0^+}p(t)u^\prime(t)=0,\\
\gamma \lim_{t\rightarrow \infty}u(t)+ \delta\lim_{t\rightarrow \infty} p(t)u^\prime(t)=0.
\end{array}%
\end{equation*}%
In order to ensure that the non-resonant scenario is considered, the condition
$$\rho=\gamma\beta+\alpha\delta+\alpha\gamma\int_0^\infty\frac{d \tau}{p(\tau)}\neq0$$
is imposed.

In \cite{EHLEE}, the authors prove the existence of multiple positive solutions for a singular Gelfand type
boundary value problem with the following second-order impulsive differential system:

\begin{equation*}
\begin{array}{l}
u^{\prime\prime}(t)\lambda h_1(t)f(u(t),v(t))=0,\;\;\; t \in (0,\,1),\;\;t\neq t_k,\\ \\
v^{\prime\prime}(t)\mu h_2(t)g(u(t),v(t))=0,\;\;\; t \in (0,\,1),\;\;t\neq t_k,\\ \\
\;\;\;\;\;\Delta u\mid_{t=t_1}=I_u(u(t_1)),\;\;\Delta v\mid_{t=t_1}=I_v(v(t_1)),\\ \\
\;\;\;\;\;\Delta u^\prime\mid_{t=t_1}=N_u(u(t_1)),\;\;\Delta v^\prime\mid_{t=t_1}=N_v(v(t_1)),\\ \\
u(0)=a\geq0,\;\;v(0)=b\geq0,\;\;u(1)=c\geq0,\;\;v(1)=d\geq0,
\end{array}
\end{equation*}
where $\lambda$, $\mu$ are positive real parameters,
$\Delta u\mid_{t=t_1}=u(t_1^+)-u(t_1^-)$, $\Delta u^\prime\mid_{t=t_1}=u^\prime(t_1^+)-u^\prime(t_1^-)$,
$f, g \in C(\mathbb{R}^2,\, (0, \infty))$,  $I_u,\, I_v \in C(\mathbb{R},\,\mathbb{R})$ satisfying
$I_u(0) =0=I_v(0)$, $N_u,\,N_v \in C(\mathbb{R},\, (-\infty,\,0])$, and $h_1,\,h_2 \in  C((0,\,1),\,(0,\,\infty))$.

Inspired by these works, we follow arguments and techniques
considered in \cite{MinhosI} and \cite{FR1}, in particular, about impulsive problems on the half-line
and second order coupled systems on the half-line, respectively.
However, it is the first time where the
existence of solutions is obtained for impulsive
coupled systems, with generalized jump conditions
in half-line and with full nonlinearities,
that depend on the unknown functions and their first
derivatives.

In particular, in the present paper, we consider the second order impulsive coupled system in
half-line composed by the differential equations
\begin{equation}
\left\{
\begin{array}{l}
u^{\prime \prime }(t)=f(t,u(t),v(t),u^{\prime }(t),v^{\prime }(t)),\text{ }
t\neq t_k, \\
v^{\prime \prime }(t)=h(t,u(t),v(t),u^{\prime }(t),v^{\prime }(t)), \text{ }t\neq \tau_j,
\end{array}%
\right.  \label{art5}
\end{equation}%
where $f,h:\left[ 0,+\infty\right[ \times \mathbb{R}^{4}\rightarrow \mathbb{R}
$ are $L^{1}$-Carath\'{e}odory functions, the boundary
conditions
\begin{equation}
\left\{
\begin{array}{c}
u(0)=A_{1},\;\,v(0)=A_{2}, \\
u^{\prime }(+\infty )=B_{1},\;\,v^{\prime }(+\infty )=B_{2},%
\end{array}%
\right.  \label{cond5}
\end{equation}%
for  $A_1, A_2, B_1, B_2\in \mathbb{R}$ and the generalized impulsive conditions%
\begin{equation}
\begin{cases}
\Delta u(t_{k})=I_{0k}(t_{k},u(t_{k}),u^{\prime }(t_{k})), \\
\Delta u^{\prime }(t_{k})=I_{1k}(t_{k},u(t_{k}),u^{\prime }(t_{k})), \\
\Delta v(\tau _{j})=J_{0j}(\tau _{j},v(\tau _{j}),v^{\prime }(\tau _{j})) \\
\Delta v^{\prime }(\tau _{j})=J_{1j}(\tau _{j},v(\tau _{j}),v^{\prime }(\tau
_{j})),%
\end{cases}
\label{imp5}
\end{equation}%
where, $k,j\in \mathbb{N}$,
\begin{equation*}
\Delta u^{(i)}(x_{k})=u^{(i)}(x_{k}^{+})-u^{(i)}(x_{k}^{-}),\;\;\Delta
v^{(i)}\left( \tau _{j}\right) =v^{(i)}(\tau _{j}^{+})-v^{(i)}(\tau
_{j}^{-}),
\end{equation*}%
$I_{ik},J_{ij}\in C(\left[ 0,+\infty \right[ \times \mathbb{R}^{2},\mathbb{R}%
),$ $i=0,1,$ with $t_{k},\tau _{j}$ fixed points such that $%
0<t_{1}<\cdots <t_{k}<\cdots $, $0<\tau _{1}<\cdots<\tau
_{j}<\cdots $ and
\begin{equation*}
\underset{k\rightarrow +\infty }{\lim }t_{k}=+\infty ,\;\,\underset{%
j\rightarrow +\infty }{\lim }\tau_{j}=+\infty .
\end{equation*}

Some points play a key role, such as: Carath\'{e}odory functions and sequences,
the equiconvergence at each impulsive moment and at infinity,
Banach spaces with weighted norms, and Schauder's fixed
point theorem, to prove the existence of solutions.

The paper is organized in the following way: In section 2 we present some
auxiliary results and definitions. Section 3 contains an existence result
for the impulsive coupled systems with generalized jump conditions
in half-line. In Section 4 the main existence theorem
is applied to a real phenomena: a model of the motion of a spring pendulum.

\section{Definitions and preliminary results}

Define \begin{equation*}u(t_{k}^{\pm }):=\underset{t\rightarrow t_{k}^{\pm }}{\lim }u(t),\;\;
v(\tau_{j}^{\pm }):=\underset{j\rightarrow \tau_{j}^{\pm }}{\lim }\tau(j),\end{equation*}%
and consider the set
\begin{equation*}
PC_{1}\left( \left[ 0,+\infty \right[\right) =\left\{
\begin{array}{c}
u:u\in C( \left[ 0,+\infty \right[ ,\mathbb{R})\text{ continuous for }t\neq
t_{k},u(t_{k})=u(t_{k}^{-}), \\
u(t_{k}^{+})\text{ exists for }k \in \mathbb{N}%
\end{array}%
\right\},
\end{equation*}%
$PC_{1}^n\left( \left[ 0,+\infty \right[\right) =\left\{ u:u^{(n)}(t) \in
PC_{1}\left( \left[ 0,+\infty \right[\right) \right\},\,n=1,2,$
\begin{equation*}
PC_{2}\left(\left[ 0,+\infty \right[ \right) =\left\{
\begin{array}{c}
v:v\in C(\left[ 0,+\infty \right[,\mathbb{R})\text{ continuous for }\tau
\neq \tau _{j},v(\tau _{j})=v(\tau _{j}^{-}), \\
v(\tau _{j}^{+})\text{ exists for }j \in \mathbb{N}%
\end{array}%
\right\},
\end{equation*}
and $PC_{2}^n\left( \left[ 0,+\infty \right[\right) =\left\{ v:v^{(n)}(j)
\in PC_{2}\left( \left[ 0,+\infty \right[\right) \right\},\,n=1,2.$\newline

Denote the space%
\begin{equation*}
X_{1}:=\left\{
x:x\in PC_{1}^{1}\left( \left[ 0,+\infty \right[ \right) :\,\underset{%
t\rightarrow +\infty }{\lim }\frac{x(t)}{1+t}\in \mathbb{R},\;\,
\underset{t\rightarrow +\infty }{\lim }x^{\prime }(t)\in \mathbb{R}%
\right\} ,
\end{equation*}

\begin{equation*}
X_{2}:=\left\{
y:y\in PC_{2}^{1}\left( \left[ 0,+\infty \right[ \right) :\,\underset{%
t\rightarrow +\infty }{\lim }\frac{y(t)}{1+t}\in \mathbb{R},\;\,
\underset{t\rightarrow +\infty }{\lim }y^{\prime }(t)\in \mathbb{R}%
\right\} ,
\end{equation*}%
and $X:=X_{1}\times X_{2}$.

In fact, $X_1$, $X_2$ and $X$ are Banach spaces with the norms

\begin{equation*}\left\Vert u\right\Vert _{X_1}=\max \left\{\left\Vert u \right\Vert_0, \left\Vert
u^\prime \right\Vert_1 \right\},\;\;\left\Vert v\right\Vert _{X_2}=\max
\left\{\left\Vert v \right\Vert_0, \left\Vert v^\prime \right\Vert_1
\right\},\end{equation*}
and \begin{equation*}\left\Vert \left( u,v\right) \right\Vert _{X}=\max \left\{
\left\Vert u\right\Vert _{X_{1}},\left\Vert v\right\Vert _{X_{2}}\right\},\end{equation*}%
respectively, where

\begin{equation*}
\left\Vert \Upsilon \right\Vert_0:=\sup_{t\in \left[ 0,+\infty \right[}
\frac{|\Upsilon(t)|}{1+t}\;\, \text{and}\;\,\left\Vert \Upsilon
\right\Vert_1:=\sup_{t\in \left[ 0,+\infty \right[} |\Upsilon(t)|.
\end{equation*}

\begin{definition}
\label{Lcar5} A function $g:\left[0,+\infty\right[\times \mathbb{R}%
^{4}\rightarrow \mathbb{R}$ is $L^{1}-$ Carath\'{e}odory if

\begin{enumerate}
\item[$i)$] for each $(x,y,z,w)\in \mathbb{R}^4$, $t\mapsto g(t,x,y,z,w)$ is
measurable on $\left[ 0,+\infty\right[;$

\item[$ii)$] for almost every $t\in \left[ 0,+\infty\right[,$ $%
(x,y,z,w)\mapsto g(t,x,y,z,w)$ is continuous on $\mathbb{R}^{4};$

\item[$iii)$] for each $\rho >0$, there exists a positive function $\phi
_{\rho }\in L^{1}\left(\left[ 0,+\infty\right[\right)$ such that, for $%
\left( x, y, z, w\right) \in \mathbb{R}^{4}$ with
\begin{equation}
\sup_{t\in \left[ 0,+\infty \right[}\left\{\frac{|x|}{1+t},\frac{|y|}{1+t},
|z|,\left\vert w\right\vert\right\} <\rho,  \label{rho5}
\end{equation}%
one has
\begin{equation}  \label{caract5}
\left\vert g(t,x,y,z,w)\right\vert \leq \phi _{\rho }(t),\text{ }a.e.\;\,%
\text{}t\in \left[ 0,+\infty\right[.
\end{equation}
\end{enumerate}
\end{definition}

\begin{definition}
\label{seq5} A sequence $(c_n)_{n \in \mathbb{N}}:\left[0,+\infty\right[%
\times \mathbb{R}^{2}\rightarrow \mathbb{R}$ is a Carath\'{e}odory sequence
if it verifies

\begin{enumerate}
\item[$i)$] for each $(a,b)\in \mathbb{R}^2$, $(a,b)\rightarrow
c_n(t,a,b)$ is continuous for all $n \in \mathbb{N}$;

\item[$ii)$] for each $\rho >0$, there are nonnegative constants $%
\chi_{n,\rho}\geq0$ with $\sum_{n=1}^{+\infty}\chi_{n,\rho}<+\infty$ such
that for $|a|<\rho(1+t)$, $t \in \left[0,+\infty\right[$ and $|b|<\rho$
we have $|c_n(t,a,b)|\leq\chi_{n,\rho}$, for every $n \in \mathbb{N}$, $%
t\in \left[ 0,+\infty\right[$.
\end{enumerate}
\end{definition}

\begin{lemma}
\label{lema1_5}Let $f,h :\left[ 0,+\infty\right[\times \mathbb{R}%
^{4}\rightarrow \mathbb{R}$ be $L^{1}-$ Carath\'{e}odory functions. Then the
system (\ref{art5}) with conditions (\ref{cond5}), (\ref{imp5}), has a solution $(u(t),\,v(t))$
expressed by%
\begin{eqnarray*}
&&u(t)=A_{1}+B_{1}t+\sum_{0<t_{k}<t<+\infty }\left[
I_{0k}(t_{k},u(t_{k}),u^{\prime }(t_{k}))+I_{1k}(t_{k},u(t_{k}),u^{\prime
}(t_{k}))\left( t-t_{k}\right) \right]  \\
&&-t\sum_{k=1}^{+\infty }I_{1k}(t_{k},u(t_{k}),u^{\prime
}(t_{k}))+\int_{0}^{+\infty }G(t,s)f(s,u(s),v(s),u^{\prime }(s),v^{\prime
}(s))ds, \\
&& \\
&&v(t)=A_{2}+B_{2}t+\sum_{0<\tau _{j}<t<+\infty }\left[ J_{0j}(\tau
_{j},v(\tau _{j}),v^{\prime }(\tau _{j}))+J_{1j}(\tau _{j},v(\tau
_{j}),v^{\prime }(\tau _{j}))\left( t-\tau _{j}\right) \right]  \\
&&-t\sum_{j=1}^{+\infty }J_{1j}(\tau _{j},v(\tau _{j}),v^{\prime }(\tau
_{j}))+\int_{0}^{+\infty }G(t,s)h(s,u(s),v(s),u^{\prime }(s),v^{\prime
}(s))ds,
\end{eqnarray*}%
\textit{where}%
\begin{equation}
G(t,s)=\left\{
\begin{array}{ll}
-t, & \;\;\;\;0\leq t\leq s\leq +\infty  \\
&  \\
-s, & \;\;\;\;0\leq s\leq t\leq +\infty .%
\end{array}%
\right.   \label{G_5}
\end{equation}
\end{lemma}
The proof follows standard techniques and it is omitted.
\begin{definition}
\label{comp} The operator $T:X\rightarrow X$ is said to be compact if $T(D)$ is
relatively compact, for $D\subseteq X$.
\end{definition}

The existence tool will be given by Schauder's fixed point theorem:

\begin{theorem}
\label{schauder} (\cite{zeidler}) Let $Y$ be a nonempty, closed, bounded and
convex subset of a Banach space $X$, and suppose that $P:Y\rightarrow Y$ is
a compact operator. Then $P$ as at least one fixed point in $Y$.
\end{theorem}

\section{Existence result}

In this section we prove the existence of solution for the problem (\ref%
{art5})-(\ref{imp5}).

\begin{theorem}
\label{main5} Let $f,h:\left[ 0,+\infty \right[ \times \mathbb{R}%
^{4}\rightarrow \mathbb{R}$ be $L^{1}-$ Carath\'{e}odory functions verifying
(\ref{caract5}) with $\Phi _{\rho }(t)$, $\Psi _{\rho }(t)$, respectively.
If $I_{0k},I_{1k},J_{0j},J_{1j}:\left[ 0,+\infty \right[ \times \mathbb{R}%
^{2}\rightarrow \mathbb{R}$ are Carath\'{e}odory sequences with nonnegative
constants $\varphi _{k,\rho }\geq 0,\psi _{k,\rho }\geq 0,\phi _{j,\rho
}\geq 0,\vartheta _{j,\rho }\geq 0$ and
\begin{equation}
\sum_{k=1}^{+\infty }\varphi _{k,\rho }<+\infty ,\;\;\sum_{k=1}^{+\infty
}\psi _{k,\rho }<+\infty ,\;\;\sum_{j=1}^{+\infty }\phi _{j,\rho }<+\infty
,\;\;\sum_{j=1}^{+\infty }\vartheta _{j,\rho }<+\infty ,
\label{marimp5.1}
\end{equation}%
such that
\begin{equation}
\begin{array}{l}
|I_{0k}(t_{k},x,y)|\leq \varphi _{k,\rho },\;\;|I_{1k}(t_{k},x,y)|\leq \psi
_{k,\rho }, \\
|J_{0k}(t_{k},x,y)|\leq \phi _{j,\rho },\;\;|J_{1k}(t_{k},x,y)|\leq
\vartheta _{j,\rho },%
\end{array}
\label{marimp5}
\end{equation}%
for $|x|<\rho (1+t)$, $|y|<\rho ,$ $t\in \left[ 0,+\infty \right[ $ then
there is at least a pair $(u,v)\in \Big(PC_1^2\left(\left[ 0,+\infty\right[ \right)\times
PC_2^2\left(\left[ 0,+\infty\right[ \right)\Big)\cap X$, solution of (\ref{art5})-(\ref{imp5}).
\end{theorem}

\begin{proof}
Define the operators $T_{1}:X\,\rightarrow X_{1}\ ,T_{2}:X\,\rightarrow X_{2}
$, and $T\,:\,X\,\rightarrow X$ by
\begin{equation}
T\left( u,v\right) =\left( T_{1}\left( u,v\right) ,T_{2}\left( u,v\right)
\right) ,  \label{T5}
\end{equation}%
with%
\begin{eqnarray*}
&&\left( T_{1}\left( u,v\right) \right) \left( t\right) =A_{1}+B_{1}t \\
&&+\sum_{0<t_{k}<t<+\infty }\left[ I_{0k}(t_{k},u(t_{k}),u^{\prime
}(t_{k}))+I_{1k}(t_{k},u(t_{k}),u^{\prime }(t_{k}))\left( t-t_{k}\right) %
\right]  \\
&&-t\sum_{k=1}^{+\infty }I_{1k}(t_{k},u(t_{k}),u^{\prime
}(t_{k}))+\int_{0}^{+\infty }G(t,s)f(s,u(s),v(s),u^{\prime }(s),v^{\prime
}(s))ds,\\
\end{eqnarray*}%
\begin{eqnarray*}
&&\left( T_{2}\left( u,v\right) \right) \left( t\right) =A_{2}+B_{2}t \\
&&+\sum_{0<\tau _{j}<t<+\infty }\left[ J_{0j}(\tau _{j},v(\tau
_{j}),v^{\prime }(\tau _{j}))+J_{1j}(\tau _{j},v(\tau _{j}),v^{\prime }(\tau
_{j}))\left( t-\tau _{j}\right) \right]  \\
&&-t\sum_{j=1}^{+\infty }J_{1j}(\tau _{j},v(\tau _{j}),v^{\prime }(\tau
_{j}))+\int_{0}^{+\infty }G(t,s)h(s,u(s),v(s),u^{\prime }(s),v^{\prime
}(s))ds,
\end{eqnarray*}%
where $G(t,s)$ is defined in (\ref{G_5}).

Let $\left( u,v\right) \in X.$

The proof will follow several steps which, for clearness, are detailed for operator $T_{1}\left(
u,v\right).$ The technique for operator $T_{2}\left( u,v\right) $ is
similar.\medskip

\textbf{Step 1}: $T$ \textit{\ is well defined and continuous on }$X$\textit{%
.}\medskip

By the Lebesgue dominated convergence theorem, (\ref{marimp5.1}) and (\ref{marimp5}),

\begin{eqnarray*}
&&\lim_{t\rightarrow+\infty} \frac{T_{1}\left( u,v\right)(t)}{1+t}=
\lim_{t\rightarrow+\infty} \left(\frac{A_1+B_1t}{1+t}\right. \\
&&+\frac{1}{1+t}\sum_{0<t_{k}<t<+\infty}\left[ I_{0k}(t_{k},u(t_{k}),u^{%
\prime }(t_{k}))+I_{1k}(t_{k},u(t_{k}),u^{\prime}(t_{k}))\left(
t-t_{k}\right) \right] \\
&&\left.-\frac{t}{1+t}\sum_{k=1}^{+\infty}I_{1k}(t_{k},u(t_{k}),u^{%
\prime}(t_{k}))\right) \\
&&+\int_{0}^{+\infty}\lim_{t\rightarrow+\infty}\frac{G(t,s)}{1+t}%
f(s,u(s),v(s),u^{\prime }(s),v^{\prime }(s))ds \\
&\leq& B_1+\sum_{0<t_{k}<t<+\infty}I_{1k}(t_{k},u(t_{k}),u^{\prime}(t_{k}))-
\sum_{k=1}^{+\infty}I_{1k}(t_{k},u(t_{k}),u^{\prime}(t_{k})) \\
&&+\int_{t}^{+\infty}|f(s,u(s),v(s),u^{\prime }(s),v^{\prime }(s))|ds \\
&\leq&
B_1+2\sum_{k=1}^{+\infty}\psi_{k,\rho}+\int_{0}^{+\infty}\Phi_\rho(s)ds<+%
\infty,
\end{eqnarray*}%
for $\rho>0$ given by (\ref{rho5}) and
\begin{eqnarray*}
\lim_{t\rightarrow+\infty} \left(T_{1} (u,v)\right)^{\prime }(t)&=&
B_1+\sum_{0<t_{k}<t<+\infty}I_{1k}(t_{k},u(t_{k}),u^{\prime}(t_{k}))\\
&&-\sum_{k=1}^{+\infty}I_{1k}(t_{k},u(t_{k}),u^{\prime}(t_{k})) \\
&&-\lim_{t\rightarrow+\infty}\int_{t}^{+\infty}f(s,u(s),v(s),u^{\prime
}(s),v^{\prime }(s))ds \\
&\leq&
B_1+2\sum_{k=1}^{+\infty}\psi_{k,\rho}+\int_{0}^{+\infty}\Phi_\rho(s)ds<+%
\infty.
\end{eqnarray*}

So, $T_{1}X\subset X_1$. Analogously, $T_{2}X\subset X_2$. Therefore, $T$ is
well defined in $X$ and, as $f$ and $h$ are $L^{1}-$ Carath\'{e}odory
functions, by Definition \ref{seq5}, $T$ is continuous.\newline

To prove that $TD$ is relatively compact, for $D\subseteq X$, it is enough
to show that:

\begin{enumerate}
\item[$i)$] $TD$ is uniformly bounded, for $D$ a bounded set in $X$;

\item[$ii)$] $TD$ is equicontinuous on each interval $]x_{k},x_{k+1}]\times]%
\tau _{j},\tau _{j+1}]$, for $k,j=1,2,...$;

\item[$iii)$] $TD$ is equiconvergent at each impulsive point and at infinity.
\end{enumerate}

\textbf{Step 2}: $TD$ \textit{\ is uniformly bounded, for }$D$\textit{\ a
bounded set in } $X$.\medskip

Let $D\subset X$ be a bounded subset. Thus, there is $\rho_1 >0$ such that,
for  $\left( u,v\right) \in D$,
\begin{eqnarray}
\left\Vert \left( u,v\right) \right\Vert _{X} &=&\max \left\{ \left\Vert
u\right\Vert _{X_{1}},\left\Vert v\right\Vert _{X_{2}}\right\}  \notag \\
&=&\max \left\{ \left\Vert u\right\Vert _{0},\left\Vert u^{\prime
}\right\Vert _{1},\left\Vert v\right\Vert _{0},\left\Vert v^{\prime
}\right\Vert _{1}\right\} <\rho _{1}.  \label{ro1_5}
\end{eqnarray}

Define
\begin{equation}
K_{i}:=\sup_{t\in \left[ 0,+\infty \right[ }\left( \frac{|A_{i}|+|B_{i}t|}{%
1+t}\right) ,\;\,i=1,2,\;\;Q(s):=\sup_{t\in \left[ 0,+\infty \right[ }\frac{%
\left\vert G(t,s)\right\vert }{1+t},\label{sup5}
\end{equation}%
with $0\leq Q(s)\leq 1,\;\forall s\in \left[ 0,+\infty \right[$.

As, $f$ is a $L^{1}$-Carath\'{e}odory function, then
\begin{eqnarray*}
\Vert T_{1}\left( u,v\right) \Vert _{0} &=&\sup_{t\in \lbrack 0,+\infty
\lbrack }\frac{|T_{1}\left( u,v\right) (t)|}{1+t} \\
&\leq &\sup_{t\in \left[ 0,+\infty \right[ }\left( \frac{|A_{1}|+|B_{1}t|}{%
1+t}\right.  \\
&&+\frac{1}{1+t}\sum_{0<t_{k}<t<+\infty }|I_{0k}(t_{k},u(t_{k}),u^{\prime
}(t_{k}))+I_{1k}(t_{k},u(t_{k}),u^{\prime }(t_{k}))\left( t-t_{k}\right) | \\
&&+\left. \frac{t}{1+t}\sum_{k=1}^{+\infty }|I_{1k}(t_{k},u(t_{k}),u^{\prime
}(t_{k}))|\right)  \\
&&+\int_{0}^{+\infty }\sup_{t\in \lbrack 0,+\infty \lbrack }\frac{|G(t,s)|}{%
1+t}|f(s,u(s),v(s),u^{\prime }(s),v^{\prime }(s))|ds \\
&\leq &K_{1}+\sup_{t\in \left[ 0,+\infty \right[ }\left( \frac{1}{1+t}%
\sum_{0<t_{k}<t<+\infty }\left[ \varphi _{k,\rho }+\psi _{k,\rho }(t-t_{k})%
\right] \right)  \\
&&+\sup_{t\in \left[ 0,+\infty \right[ }\left( \frac{t}{1+t}%
\sum_{k=1}^{+\infty }\psi _{k,\rho }\right) +\int_{0}^{+\infty }Q(s)\Phi
_{\rho }(s)ds \\
&\leq &K_{1}+\sum_{k=1}^{+\infty }\varphi _{k,\rho }+2\sum_{k=1}^{+\infty
}\psi _{k,\rho }+\int_{0}^{+\infty }Q(s)\Phi _{\rho }(s)ds<+\infty ,\forall
\left( u,v\right) \in D,
\end{eqnarray*}%
and

\begin{eqnarray*}
\Vert \left(T_{1}\left( u,v\right)\right)^\prime\Vert_1 &=& \sup_{t
\in[0,+\infty[} \vert \left(T_{1}\left( u,v\right)\right)^\prime(t)\vert \\
&\leq& |B_1|+\sup_{t
\in[0,+\infty[}\sum_{0<t_{k}<t<+\infty}|I_{1k}(t_{k},u(t_{k}),u^{%
\prime}(t_{k}))|\\
&&+\sum_{k=1}^{+\infty}|I_{1k}(t_{k},u(t_{k}),u^{\prime}(t_{k}))| \\
&&+\sup_{t \in[0,+\infty[}\int_{t}^{+\infty}|f(s,u(s),v(s),u^{\prime
}(s),v^{\prime }(s))|ds \\
&\leq&
|B_1|+2\sum_{k=1}^{+\infty}\psi_{k,\rho}+\int_{0}^{+\infty}\Phi_\rho(s)ds<+%
\infty.
\end{eqnarray*}

Therefore, $T_1D$ is bounded and by similar arguments, $T_2D$ is also
bounded. Furthermore, $\left\Vert T( u,v)\right\Vert_{X}<+\infty$, that is $%
TD$\ is uniformly bounded on $X$.\newline

\textbf{Step 3}: $TD$\textit{\ is equicontinuous on each interval} $%
]x_{k},x_{k+1}]\times ]\tau _{j},\tau _{j+1}]$, \textit{\ that is}, $T_{1}D$
\textit{\ is equicontinuous on each interval }$]x_{k},x_{k+1}],$\textit{\
for }$k\in  \mathbb{N}$, $0<t_{1}<\cdots <t_{k}<\cdots ,$\textit{\ and }$T_{2}D$%
\textit{\ is equicontinuous on each interval }$]\tau _{j},\tau _{j+1}]$,%
\textit{\ for }$j\in  \mathbb{N}$\textit{\ and }$0<\tau _{1}<\cdots <\tau
_{j}<\cdots .$\medskip

Consider $I\subseteq ]x_{k},x_{k+1}]$ and $\iota _{1},\iota _{2}\in I$ such
that $\iota _{1}\leq \iota _{2}$. For $(u,v)\in D$, we have%
\begin{eqnarray*}
&&\lim_{\iota _{1}\rightarrow \iota _{2}}\left\vert \frac{T_{1}\left(
u,v\right) (\iota _{1})}{1+\iota _{1}}-\frac{T_{1}\left( u,v\right) (\iota
_{2})}{1+\iota _{2}}\right\vert \leq \lim_{\iota _{1}\rightarrow \iota
_{2}}\left\vert \frac{A_{1}+B_{1}\iota _{1}}{1+\iota _{1}}-\frac{%
A_{1}+B_{1}\iota _{2}}{1+\iota _{2}}\right\vert \\
&&+\left\vert \frac{1}{1+\iota _{1}}\sum_{0<t_{k}<\iota _{1}}\left[
I_{0k}(t_{k},u(t_{k}),u^{\prime }(t_{k}))+I_{1k}(t_{k},u(t_{k}),u^{\prime
}(t_{k}))\left( \iota _{1}-t_{k}\right) \right] \right. \\
&&-\frac{1}{1+\iota _{2}}\sum_{0<t_{k}<\iota _{2}}\left[
I_{0k}(t_{k},u(t_{k}),u^{\prime }(t_{k}))+I_{1k}(t_{k},u(t_{k}),u^{\prime
}(t_{k}))\left( \iota _{2}-t_{k}\right) \right] \\
&&\left. -\frac{\iota _{1}}{1+\iota _{1}}\sum_{k=1}^{+\infty
}I_{1k}(t_{k},u(t_{k}),u^{\prime }(t_{k}))+\frac{\iota _{2}}{1+\iota _{2}}%
\sum_{k=1}^{+\infty }I_{1k}(t_{k},u(t_{k}),u^{\prime }(t_{k}))\right\vert \\
&&+\int_{0}^{+\infty }\lim_{\iota _{1}\rightarrow \iota _{2}}\left\vert
\frac{G(\iota _{1},s)}{1+\iota _{1}}-\frac{G(\iota _{2},s)}{1+\iota _{2}}%
\right\vert \left\vert f(s,u(s),v(s),u^{\prime }(s),v^{\prime
}(s))\right\vert ds=0,
\end{eqnarray*}%
as $\iota _{1}\rightarrow \iota _{2}$, and%
\begin{eqnarray*}
&&\lim_{\iota _{1}\rightarrow \iota _{2}}\left\vert \left( T_{1}\left(
u,v\right) (\iota _{1})\right) ^{\prime }-\left( T_{1}\left( u,v\right)
(\iota _{2})\right) ^{\prime }\right\vert \\
&\leq &\lim_{\iota _{1}\rightarrow \iota _{2}}\left\vert \sum_{0<t_{k}<\iota
_{1}}I_{1k}(t_{k},u(t_{k}),u^{\prime }(t_{k}))-\sum_{0<t_{k}<\iota
_{2}}I_{1k}(t_{k},u(t_{k}),u^{\prime }(t_{k}))\right. \\
&&\left. -\int_{\iota _{1}}^{+\infty }f(s,u(s),v(s),u^{\prime }(s),v^{\prime
}(s))ds+\int_{\iota _{2}}^{+\infty }f(s,u(s),v(s),u^{\prime }(s),v^{\prime
}(s))ds\right\vert \\
&\leq &\lim_{\iota _{1}\rightarrow \iota _{2}}\sum_{\iota _{1}<t_{k}<\iota
_{2}}|I_{1k}(t_{k},u(t_{k}),u^{\prime }(t_{k}))|+\int_{\iota _{1}}^{\iota
_{2}}|f(s,u(s),v(s),u^{\prime }(s),v^{\prime }(s))|ds \\
&\leq &\lim_{\iota _{1}\rightarrow \iota _{2}}\sum_{\iota _{1}<t_{k}<\iota
_{2}}\psi _{k,\rho }+\int_{\iota _{1}}^{\iota _{2}}\Phi _{\rho }(s)ds=0.
\end{eqnarray*}

Therefore, $T_{1}D$ is equicontinuous on $X_{1}$. Similarly, we can show
that $T_{2}D$ is equicontinuous on $X_{2}$, too. Thus, $TD$ is
equicontinuous on $X$.\medskip

\textbf{Step 4}: $TD$ \textit{is equiconvergent at each impulsive point and
at infinity, that is} $T_1D$, \textit{is equiconvergent at} $t=t^+_i$, $%
i=1,2,...,$ \textit{and at infinity}, and $T_2D$, \textit{is
equiconvergent at} $\tau=\tau^{+}_l$, $l=1,2$,...,\textit{and at infinity }%
.\medskip

First, let us prove the equiconvergence at $t=t^+_i$, for $i=1,2,...$. The
proof for the equiconvergence at $\tau=\tau^{+}_l$, for $l=1,2,...,$ is
analogous.

Thus, it follows%
\begin{eqnarray*}
&&\left\vert \frac{T_{1}\left( u,v\right) (t)}{1+t}-\lim_{t\rightarrow
t_{i}^{+}}\frac{T_{1}\left( u,v\right) (t)}{1+t}\right\vert \leq \left\vert
\frac{A_{1}+B_{1}t}{1+t}-\frac{A_{1}+B_{1}t_{i}}{1+t_{i}}\right\vert \\
&&+\left\vert \frac{1}{1+t}\sum_{0<t_{k}<t<+\infty }\left[
I_{0k}(t_{k},u(t_{k}),u^{\prime }(t_{k}))+I_{1k}(t_{k},u(t_{k}),u^{\prime
}(t_{k}))\left( t-t_{k}\right) \right] \right. \\
&&\left. -\frac{1}{1+t_{i}}\sum_{0<t_{k}<t_{i}^{+}}\left[
I_{0k}(t_{k},u(t_{k}),u^{\prime }(t_{k}))+I_{1k}(t_{k},u(t_{k}),u^{\prime
}(t_{k}))\left( t_{i}-t_{k}\right) \right] \right\vert \\
&&+\left\vert -\frac{t}{1+t}\sum_{k=1}^{+\infty
}I_{1k}(t_{k},u(t_{k}),u^{\prime }(t_{k}))+\frac{t_{i}}{1+t_{i}}%
\sum_{k=1}^{+\infty }I_{1k}(t_{k},u(t_{k}),u^{\prime }(t_{k}))\right\vert \\
&&+\int_{0}^{+\infty }\left\vert \frac{G(t,s)}{1+t}-\frac{G(t,s)}{1+t_{i}}%
\right\vert \Phi _{\rho }(s)ds\rightarrow 0,
\end{eqnarray*}%
uniformly on $(u,v)\in D$, as $t\rightarrow t_{i}^{+}$, for $i=1,2,...$ and%
\begin{eqnarray*}
&&\left\vert \left( T_{1}\left( u,v\right) (t)\right) ^{\prime
}-\lim_{t\rightarrow t_{i}^{+}}\left( T_{1}\left( u,v\right) (t)\right)
^{\prime }\right\vert \\
&=&\left\vert \sum_{0<t_{k}<t<+\infty }I_{1k}(t_{k},u(t_{k}),u^{\prime
}(t_{k}))-\sum_{0<t_{k}<t_{i}^{+}}I_{1k}(t_{k},u(t_{k}),u^{\prime
}(t_{k}))\right. \\
&&\left. -\int_{t}^{+\infty }f(s,u(s),v(s),u^{\prime }(s),v^{\prime
}(s))ds+\int_{t_{i}}^{+\infty }f(s,u(s),v(s),u^{\prime }(s),v^{\prime
}(s))ds\right\vert \\
&\leq &\left\vert \sum_{0<t_{k}<t<+\infty }I_{1k}(t_{k},u(t_{k}),u^{\prime
}(t_{k}))-\sum_{0<t_{k}<t_{i}^{+}}I_{1k}(t_{k},u(t_{k}),u^{\prime
}(t_{k}))\right\vert \\
&&+\left\vert -\int_{t}^{+\infty }f(s,u(s),v(s),u^{\prime }(s),v^{\prime
}(s))ds+\int_{t_{i}}^{+\infty }f(s,u(s),v(s),u^{\prime }(s),v^{\prime
}(s))ds\right\vert \\
&\leq &\left\vert \sum_{0<t_{k}<t<+\infty }I_{1k}(t_{k},u(t_{k}),u^{\prime
}(t_{k}))-\sum_{0<t_{k}<t_{i}^{+}}I_{1k}(t_{k},u(t_{k}),u^{\prime
}(t_{k}))\right\vert \\
&&+\int_{t_{i}}^{t}\phi _{\rho }(s)ds\rightarrow 0,
\end{eqnarray*}%
uniformly on $(u,v)\in D$, as $t\rightarrow t_{i}^{+}$, for $i=1,2,...$.

Therefore, $T_1D$ is equiconvergent at each point $t=t^+_i$, for $%
i=1,2,... $. Analogously, it can be proved that $T_2D$ is equiconvergent
at each point $\tau=\tau^{+}_l$, for $l=1,2,...,$.

So, $TD$ is equiconvergent at each impulsive point.

To prove the equiconvergence at infinity, for the operator $T_{1}$, we have%
\begin{eqnarray*}
&&\left\vert \frac{T_{1}\left( u,v\right) (t)}{1+t}-\lim_{t\rightarrow
+\infty }\frac{T_{1}\left( u,v\right) (t)}{1+t}\right\vert \leq \left\vert
\frac{A_{1}+B_{1}t}{1+t}-B_{1}\right\vert  \\
&&+\left\vert \frac{1}{1+t}\sum_{0<t_{k}<t<+\infty }\left[
I_{0k}(t_{k},u(t_{k}),u^{\prime }(t_{k}))+I_{1k}(t_{k},u(t_{k}),u^{\prime
}(t_{k}))\left( t-t_{k}\right) \right] \right.  \\
&&\left. -\lim_{t\rightarrow +\infty }\frac{1}{1+t}\sum_{0<t_{k}<t<+\infty }%
\left[ I_{0k}(t_{k},u(t_{k}),u^{\prime
}(t_{k}))+I_{1k}(t_{k},u(t_{k}),u^{\prime }(t_{k}))\left( t-t_{k}\right) %
\right] \right\vert  \\
&&+\left\vert -\frac{t}{1+t}\sum_{k=1}^{+\infty
}I_{1k}(t_{k},u(t_{k}),u^{\prime }(t_{k}))+\lim_{t\rightarrow +\infty }\frac{%
t}{1+t}\sum_{k=1}^{+\infty }I_{1k}(t_{k},u(t_{k}),u^{\prime
}(t_{k}))\right\vert  \\
&&+\int_{0}^{+\infty }\left\vert \frac{G(t,s)}{1+t}-\lim_{t\rightarrow
+\infty }\frac{G(t,s)}{1+t}\right\vert \left\vert f(s,u(s),v(s),u^{\prime
}(s),v^{\prime }(s))ds\right\vert  \\
&\leq &\left\vert \frac{A_{1}+B_{1}t}{1+t}-B_{1}\right\vert  \\
&&+\left\vert \frac{1}{1+t}\sum_{0<t_{k}<t<+\infty }\left[
I_{0k}(t_{k},u(t_{k}),u^{\prime }(t_{k}))+I_{1k}(t_{k},u(t_{k}),u^{\prime
}(t_{k}))\left( t-t_{k}\right) \right] \right.  \\
&&\left. -\sum_{k=1}^{+\infty }I_{1k}(t_{k},u(t_{k}),u^{\prime
}(t_{k}))\right\vert  \\
&&+\left\vert -\frac{t}{1+t}\sum_{k=1}^{+\infty
}I_{1k}(t_{k},u(t_{k}),u^{\prime }(t_{k}))+\sum_{k=1}^{+\infty
}I_{1k}(t_{k},u(t_{k}),u^{\prime }(t_{k}))\right\vert  \\
&&+\int_{0}^{+\infty }\left\vert \frac{G(t,s)}{1+t}-\lim_{t\rightarrow
+\infty }\frac{G(t,s)}{1+t}\right\vert \Phi _{\rho }(s)ds\rightarrow 0,
\end{eqnarray*}%
uniformly on $(u,v)\in D$, as $t\rightarrow +\infty $.

Analogously,%
\begin{eqnarray*}
&&\left\vert \left( T_{1}\left( u,v\right) (t)\right) ^{\prime
}-\lim_{t\rightarrow +\infty }\left( T_{1}\left( u,v\right) (t)\right)
^{\prime }\right\vert  \\
&=&\left\vert \sum_{0<t_{k}<t<+\infty }I_{1k}(t_{k},u(t_{k}),u^{\prime
}(t_{k}))-\sum_{k=1}^{+\infty }I_{1k}(t_{k},u(t_{k}),u^{\prime
}(t_{k}))\right.  \\
&&-\int_{t}^{+\infty }f(s,u(s),v(s),u^{\prime }(s),v^{\prime }(s))ds \\
&&-\lim_{t\rightarrow +\infty }\sum_{0<t_{k}<t<+\infty
}I_{1k}(t_{k},u(t_{k}),u^{\prime }(t_{k}))+\lim_{t\rightarrow +\infty
}\sum_{k=1}^{+\infty }I_{1k}(t_{k},u(t_{k}),u^{\prime }(t_{k})) \\
&&\left. +\lim_{t\rightarrow +\infty }\int_{t}^{+\infty
}f(s,u(s),v(s),u^{\prime }(s),v^{\prime }(s))ds\right\vert  \\
&\leq &\left\vert \sum_{k=0}^{+\infty }I_{1k}(t_{k},u(t_{k}),u^{\prime
}(t_{k}))-\sum_{k=1}^{+\infty }I_{1k}(t_{k},u(t_{k}),u^{\prime
}(t_{k}))\right\vert  \\
&&+\int_{t}^{+\infty }|f(s,u(s),v(s),u^{\prime }(s),v^{\prime }(s))|ds \\
&\leq &\left\vert \sum_{0<t_{k}<t<+\infty }I_{1k}(t_{k},u(t_{k}),u^{\prime
}(t_{k}))-\sum_{k=1}^{+\infty }I_{1k}(t_{k},u(t_{k}),u^{\prime
}(t_{k}))\right\vert  \\
&&+\int_{t}^{+\infty }\Phi _{\rho }(s)ds\rightarrow 0,
\end{eqnarray*}%
uniformly on $(u,v)\in D$, as $t\rightarrow +\infty $.

So, $T_1D$ is equiconvergent at $+\infty$. Following the same arguments, $%
T_2D$ is equiconvergent at $+\infty$, too. Therefore, $TD$ is equiconvergent
at $+\infty$. \newline

Therefore, $TD$ is relatively compact and, by Definition \ref{comp}, $T$ is compact.
\newline

In order to apply Theorem \ref{main5}, we need the next step:\newline

\textbf{Step 5}: $T\Omega\subset \Omega$ \textit{for some} $\Omega\subset X$
\textit{a closed and bounded set.}\medskip

Consider
\begin{equation*}
\Omega :=\left\{ (u,v)\in E:\left\Vert (u,v)\right\Vert _{X}\leq \rho
_{2}\right\} ,
\end{equation*}%
with $\rho _{2}>0$ such that%
\begin{equation}
\rho _{2}:=\max \left\{
\begin{array}{c}
\rho _{1},\;\,K_{1}+\sum_{k=1}^{+\infty }\varphi _{k,\rho
}+2\sum_{k=1}^{+\infty }\psi _{k,\rho }+\int_{0}^{+\infty }Q(s)\Phi _{\rho
}(s)ds, \\
\\
K_{2}+\sum_{k=1}^{+\infty }\phi _{j,\rho }+2\sum_{k=1}^{+\infty }\vartheta
_{j,\rho }+\int_{0}^{+\infty }Q(s)\Psi _{\rho }(s)ds, \\
\\
|B_{1}|+2\sum_{k=1}^{+\infty }\psi _{k,\rho }+\int_{0}^{+\infty }\Phi _{\rho
}(s)ds, \\
\\
\ |B_{2}|+2\sum_{k=1}^{+\infty }\vartheta _{j,\rho }+\int_{0}^{+\infty }\Psi
_{\rho }(s)ds,%
\end{array}%
\right\} ,  \label{ro2}
\end{equation}%
with $\rho _{1}$ given by (\ref{ro1_5}). According to Step 2 and $K_{1}$, $%
K_{2}$ and $Q(s)$ given by (\ref{sup5}), we have
\begin{eqnarray*}
\left\Vert T(u,v)\right\Vert _{X} &=&\left\Vert \left(
T_{1}(u,v),T_{2}\left( u,v\right) \right) \right\Vert _{X_{1}} \\
&=&\max \left\{ \left\Vert T_{1}\left( u,v\right) \right\Vert
_{X},\;\left\Vert T_{2}\left( u,v\right) \right\Vert _{X_{2}}\right\}  \\
&=&\max \left\{ \left\Vert T_{1}\left( u,v\right) \right\Vert
_{0},\;\left\Vert \left( T_{1}\left( u,v\right) \right) ^{\prime
}\right\Vert _{1},\;\left\Vert T_{2}\left( u,v\right) \right\Vert
_{0},\;\left\Vert \left( T_{2}\left( u,v\right) \right) ^{\prime
}\right\Vert _{1}\right\}  \\
&\leq &\rho _{2}.
\end{eqnarray*}

So, $T\Omega\subset \Omega$, and by Theorem \ref{schauder}, the operator $%
T\left( u,v\right) =\left( T_{1}\left( u,v\right) ,T_{2}\left( u,v\right)
\right)$, has a fixed point $(u, v)$.

By standard techniques, and Lemma \ref{lema1_5}, it can be
shown that this fixed point is a solution
of problem (\ref{art5})-(\ref{imp5}).
\end{proof}

\section{Motion of a spring pendulum}

Consider the motion of the spring pendulum of a mass
attached to one end of a spring and the other
end attached to the ceiling. By  \cite{NSS}, we
represent this motion by the system,

\begin{equation}
\left\{
\begin{array}{l}
l^{\prime\prime}(t)=\frac{1}{t^3}\left(l(t)\theta^\prime(t)-g\cos(\theta(t))-\frac{k}{m}(l(t)-l_0)\right),\;t\in \left[ 0,+\infty \right[ \\
\\
\theta^{\prime\prime}(t)=\frac{1}{t^3}\left(\frac{-gl(t)\sin(\theta(t))-2l(t)l^\prime(t)\theta^\prime(t)}{l^2(t)}\right),
\end{array}%
\right.\label{Ap5.1}
\end{equation}

where:
\begin{itemize}
\item $l(t)$, $l_0$ are the length at time $t$ and the natural length of the spring, respectively;

\item $\theta(t)$ is the angle between the pendulum and the
vertical;

\item $m$, $k$, $g$ are the mass, the spring constant and gravitational force, respectively;
\end{itemize}
together with the boundary conditions

\begin{equation}
\left\{
\begin{array}{c}
l(0)=0,\;\,\theta(0)=0, \\
l^{\prime }(+\infty)=B_1,\;\,\theta^{\prime }(+\infty)=B_2,%
\end{array}%
\right.\label{B5}
\end{equation}%
with  $B_1, B_2 \in \left[0,\,\pi\right]$, and the generalized impulsive conditions

\begin{equation}
\begin{cases}
\Delta \theta_1(t_{k})=\frac{1}{k^3}(\alpha_1\theta_1(t_{k})+\alpha_2\theta_1^{\prime }(t_{k})),\\
\\
\Delta \theta_1^{\prime }(t_{k})=\frac{1}{k^\beta}(\alpha_3\theta_1(t_{k})+\alpha_4\theta_1^{\prime }(t_{k})),\;\beta\geq3,\\
 \\
\Delta \theta_2(\tau _{j})=\frac{1}{j^\gamma}(\alpha_5\theta_2(\tau_{j})+\alpha_6\theta_2^{\prime }(\tau_{j})),\;\gamma\geq3\\
 \\
\Delta \theta_2^{\prime }(\tau _{j})=\frac{1}{j^3}(\alpha_7\theta_2(\tau_{j})+\alpha_8\theta_2^{\prime }(\tau_{j})),
\end{cases}
\label{BI5}
\end{equation}
with $\alpha_i \in \mathbb{R}$, $i=1,2,...,8$ and for $k,j \in \mathbb{N}$, $0<t_{1}<\cdots <t_{k}<\cdots $, $0<\tau _{1}<\cdots \tau
_{j}<\cdots $.

\begin{figure}[h]
    \center
    \includegraphics[width=6cm]{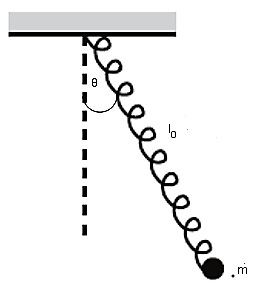}
    \label{f7}
    \caption{Motion of the spring pendulum.}
\end{figure}

The system (\ref{Ap5.1})-(\ref{BI5}) is a particular case of the problem (\ref{art5})-(\ref{imp5}), with
\begin{eqnarray*}
f(t,x,y,z,w)&=& x w-g\cos(y)-\frac{k}{m}(x-l_0),\\
h(t,x,y,z,w)&=&\frac{-gx\sin(y)-2xzw}{x^2},\\
I_{0k}(t_{k},x,z)&=&\frac{1}{k^3}(\alpha_1x+\alpha_2z),\;\;I_{1k}(t_{k},x,z)=\frac{1}{k^\beta}(\alpha_3x+\alpha_4z),\;\beta\geq3,\\
J_{0j}(\tau _{j},y, w)&=&\frac{1}{j^\gamma}(\alpha_5y+\alpha_6w),\;\gamma\geq3,
\;\;J_{1j}(\tau _{j},y, w)=\frac{1}{j^3}(\alpha_7y+\alpha_8w),
\label{f5,h5}
\end{eqnarray*}
with $t_k=k,\,\tau_j=j,\;k,j\in \mathbb{N}$, $\alpha_i \in \mathbb{R}$, $i=1,2,...,8$.

In fact, $f$ and $h$ are $L^{1}$-Carath\'{e}odory functions, with

\begin{equation*}
f(t,x,y,z,w)\leq\frac{1}{t^3}\left(\rho^2(1+t)+g+\frac{k}{m}(\rho(1+t)+l_0)\right):=\phi_\rho(t),
\end{equation*}
\begin{eqnarray*}
h(t,x,y,z,w)&\leq&\frac{1}{t^3}\left(\frac{g\rho(1+t)+2\rho^3(1+t)}{l^2(t)}\right)\\
&\leq&\frac{1}{t^3}\left(\frac{g\rho(1+t)+2\rho^3(1+t)}{l^2(t)}\right)\\
&\leq&\frac{1}{t^3}\cdot\frac{1}{\left(\min_{t\in \left[ 0,+\infty \right[}l(t)\right)^2}\left(g\rho(1+t)+2\rho^3(1+t)\right)\\
&:=&\varphi_\rho(t),
\end{eqnarray*}
and $\phi_\rho(t)$, $\varphi_\rho(t)$ verifying Definition \ref{Lcar5}.

Besides $I_{0k}$, $I_{1k}$, $J_{0j}$, $J_{1j}$, are Carath\'{e}odory sequences and verify (\ref{marimp5}), as

\begin{eqnarray*}
I_{0k}(t_{k},x,z)\leq\frac{\rho[(\alpha_1(1+k)+\alpha_2)]}{k^3},
\;\;I_{1k}(t_{k},x,z)\leq\frac{\rho[(\alpha_3(1+k)+\alpha_4)]}{k^\beta},\;\beta\geq3,\\
J_{0j}(\tau _{j},y, w)\frac{\rho[(\alpha_5(1+k)+\alpha_6)]}{j^\gamma},\;\gamma\geq3,
\;\;J_{1j}(\tau _{j},y, w)\leq\frac{\rho[(\alpha_7(1+k)+\alpha_8)]}{j^3}.
\label{f5,h51}
\end{eqnarray*}

So, by Theorem \ref{main5}, there is at least a pair $(l,\,\theta)\in\Big(PC_1^2\left(\left[ 0,+\infty\right[ \right)\times
PC_2^2\left(\left[ 0,+\infty\right[ \right)\Big)\cap X$,
solution of problem (\ref{Ap5.1}), (\ref{B5}) and (\ref{BI5}).\\\\

\noindent\textbf{Acknowledgements}

\noindent The authors thank the referees for suggestions and comments 
to further improve this work. This work was supported by National Funds through Funda\c{c}%
\~{a}o para a Ci\^{e}ncia e Tecnologia (FCT), project UID/MAT/04674/2013 (CIMA).\\

\noindent\textbf{Competing interests}

\noindent The authors declare that there is no conflict of interests regarding the
publication of this paper.\\

\noindent\textbf{Authors' contributions}

\noindent All authors read and approved the final manuscript.

\end{document}